\documentclass[12pt,leqno]{article}
\usepackage[a4paper, margin=23.8mm]{geometry}
\usepackage{amssymb}
\usepackage{amsmath}
\usepackage{amsthm}
\usepackage{stmaryrd}
\usepackage{mathtools,mathrsfs}
\usepackage[pagebackref]{hyperref}
\hypersetup{citecolor=purple, linkcolor=blue, colorlinks=true}
\usepackage{booktabs}
\usepackage{enumerate}
\usepackage{csquotes}
\usepackage{leftidx}
\usepackage{tikz}
\usepackage{url}
\usepackage{comment}
\usepackage[T2A,T1]{fontenc}
\usepackage{xfrac}
\usepackage{authblk}  

\newtheorem{theorem}{Theorem}[section]
\newtheorem{proposition}[theorem]{Proposition}
\newtheorem*{proposition*}{Proposition}

\theoremstyle{definition}
\newtheorem{definition}[theorem]{Definition}
\newtheorem{remark}[theorem]{Remark}

\newcommand{\eps}{\varepsilon}

\newcommand{\Aut}{\operatorname{Aut}}
\newcommand{\Alt}{\operatorname{Alt}}
\newcommand{\PSL}{\operatorname{PSL}}

\newcommand{\Sym}{\operatorname{Sym}}
\renewcommand{\le}{\leqslant}

\newcommand{\id}{\operatorname{id}}
\newcommand{\fix}{\operatorname{fix}}
\newcommand{\Out}{\operatorname{Out}}
\newcommand{\Gal}{\operatorname{Gal}}
\newcommand{\GL}{\operatorname{GL}}
\newcommand{\PSp}{\operatorname{PSp}}
\newcommand{\im}{\operatorname{im}}

\newcommand{\Mod}[1]{\ (\textup{mod}\ #1)}
\renewcommand{\P}{\mathfrak{P}}
\newcommand{\F}{\mathbb{F}}

\newcommand{\conj}{\operatorname{conj}}
\newcommand{\IN}{\mathbb{N}}
\newcommand{\Sp}{\operatorname{Sp}}
\newcommand{\IF}{\mathbb{F}}
\newcommand{\Stab}{\operatorname{Stab}}

\newcommand{\Hol}{\operatorname{Hol}}
\newcommand{\IZ}{\mathbb{Z}}
\newcommand{\SL}{\operatorname{SL}}
\newcommand{\Fcal}{\mathcal{F}}
\newcommand{\GammaL}{\Gamma{\rm L}}
\newcommand{\tr}{\mathrm{tr}}
\renewcommand{\b}{\mathfrak{b}}
\newcommand{\Q}{\operatorname{Q}}
\newcommand{\Z}{\operatorname{Z}}
\newcommand{\GammaSp}{\Gamma{\rm Sp}}

\newcommand\scalemath[2]{\scalebox{#1}{\mbox{\ensuremath{\displaystyle #2}}}}
\newcommand{\coloneq}{\vcentcolon=}
\newcommand{\eqcolon}{=\vcentcolon}

\begin{document}

\title{Finite $2$-groups with exactly three\\ automorphism orbits}

 \author[1]{Alexander Bors\footnote{Supported by the Austrian Science Fund (FWF), project J4072-N32 ``Affine maps on finite groups''}}
 \affil[1]{\small Johannes Kepler University (Linz)
   \href{mailto:Alexander.Bors@gmx.at}{Alexander.Bors@gmx.at}
 	}

 \author[2]{Stephen P. Glasby\footnote{Supported by the Australian Research Council (ARC) Discovery Project DP190100450}}
 \affil[2]{\small Center for the Mathematics of Symmetry and Computation,
   University of Western Australia, Boorloo Perth 6009, Australia.
   \href{mailto:Stephen.Glasby@uwa.edu.au}{Stephen.Glasby@uwa.edu.au}
 	}

\date{\today}

\maketitle

\abstract{We give a complete classification of the finite $2$-groups $G$ for
  which the automorphism group $\Aut(G)$ acting naturally on $G$ has three
  orbits. There are two infinite families and one additional group, of
  order $2^9$. All of them are Suzuki $2$-groups, and they appear (in a
  different context) in an earlier classification of Dornhoff.
}

\section{Introduction}\label{sec1}

\subsection{Motivation and main result}\label{subsec1P1}

The automorphism group $\Aut(A)$ of an algebraic object $A$ acts on $A$, and
the objects $A$ with `fewest possible' $\Aut(A)$-orbits are of great
mathematical interest; being `maximally homogeneous', in some sense.
We call $A$ a \emph{$k$-orbit} object if $\Aut(A)$ has~$k$ orbits on~$A$.
There has been much work classifying $k$-orbit objects for small~$k$.
For example, a 1-orbit group is trivial, a finite 2-orbit group is elementary
abelian, and the finite 3-orbit groups $G$ with $\Aut(G)$ solvable have been
classified by Dornhoff~\cite{Dornhoff}. The purpose of this paper is
to classify all finite 3-orbit 2-groups.

A more general problem is to classify the
permutation subgroups of the symmetric group $\Sym(X)$
with precisely $k$ orbits on~$X$. Clearly $G$ is a $k$-orbit group
precisely when $\Aut(G)\le\Sym(G)$ has $k$ orbits on $X=G$.
The groups that act transitively (and faithfully) on the
set $X=V\setminus\{0\}$ of non-zero vectors of a finite vector
space $V=(\F_q)^n$ have been classified. Huppert
classified the solvable groups with this property, and Hering~\cite{Her85a}
dealt with the non-solvable groups: for a clear statement and proof see~\cite[Appendix~1]{Lie87a}.

The most difficult case of classifying permutation subgroups
of $\Sym(X)$ with $k$ orbits is inevitably classifying the solvable ones.
For example, the solvable subgroups of $\GL_n(\F_q)\rtimes\Aut(\F_q)$ that
act transitively on the non-zero
vectors $(\F_q)^n\setminus\{0\}$ are subgroups of $\GammaL_1(\F_{q^n})$ or known
small examples with $q$ odd,
and classifying the former leads to unsolved problems in
geometry~\cite{Foulser69}. Before
saying more, we recall that the \emph{rank} of a permutation
group $G\le\Sym(X)$ is the number of $G$-orbits on $X\times X$. The holomorph of $G$, namely
$\Hol(G)=\Aut(G)\ltimes G$, acts faithfully on $X=G$ via $x^{\alpha g}=x^{\alpha} g$.
Further, the stabiliser $\Hol(G)_1$ of the identity element $1\in G$
is $\Aut(G)$ (as $1^{\alpha g}=1\Leftrightarrow g=1$) and $G$ acts regularly on $X=G$.
Importantly, $G$ is a $k$-orbit group precisely when $\Hol(G)\le\textup{Sym}(G)$ has rank $k$, and $\Hol(G)$ is imprimitive if $G$ is not characteristically simple
because the cosets of a characteristic subgroup $H$ of $G$ give
a system of imprimitivity, as $(Hx)^{\alpha g}=Hx^{\alpha}g$ for
all $\alpha g\in\Aut(G)\ltimes G$.

The finite permutation groups of rank 2 (also called \emph{2-transitive groups})
have been classified subject to the difficult problem of determining
the (solvable) rank-2 subgroups of
$A=\F_{q^n}\rtimes\GammaL_1(\F_{q^n})\le\Sym(\F_{q^n})$.
All subgroups of $A$ containing the
subgroup $N=\F_{q^n}\rtimes\F_{q^n}^\times={\rm AGL}_1(\F_{q^n})$ are 2-transitive; however, some subgroups not containing $N$
are also 2-transitive.
A rank-3 permutation group is transitive and can be imprimitive or primitive.
The solvable imprimitive rank-3 groups are classified in~\cite{Dornhoff,DMSSY}, 
the quasi-primitive rank-3 groups are classified in~\cite{AGLPP}, and
the innately transitive rank-3 groups in~\cite{BDP}.

Progress classifying the $k$-orbit $p$-groups (groups of
prime-power order) has been slow because:
(1) the automorphism groups of a finite $p$-groups 
are poorly understood and vary considerably, and
(2) there are  a huge number of $p$-groups of order $p^n$ when $n$ is large.
A major problem is the classification of finite 3-orbit $p$-groups.
The abelian ones are the homocyclic groups $(C_{p^2})^n$, and if $p$ is odd,
these are the only 3-orbit $p$-groups of exponent $p^2$ by Shult's result~\cite[Corollary~3]{Shu69a}.
Our preprint~\cite[Theorem~1.1]{BG} classifies all finite $3$-orbit
$2$-groups. These turn out to be the $2$-groups in Dornhoff's
list~\cite[Theorem (iv),\,(v),\,(viii)]{Dornhoff} as we prove that
$\Aut(G)$ must be solvable, see Theorem~\ref{normalizeTheo} and Proposition~\ref{normalizeProp} below. This paper is a paired-down version of~\cite{BG}.

Many authors have studied finite groups $G$ with the
property that elements $g,h\in G$ of the same order are `conjugate' or `fused'
in $\Aut(G)$.
These groups are partially classified in~\cite{LP}: the $2$-groups with
this property were not classified.
We were motivated to classify finite 3-orbit 2-groups precisely because
this is difficult.

In a recent beautiful paper, \cite{LZ}, Li and Zhu solve
a problem which is more general than ours:
classifying the $p$-groups $G$ for which $\Aut(G)$ is transitive on the set
of elements of order $p$. Shult's \cite[Corollary 3]{Shu69a} solves the
case $p>2$. The authors were unaware of our work, and
show in \cite[Theorem~1.1]{LZ} that the only examples are the homocyclic groups
$(C_{p^2})^n$, $p$ odd, by~\cite[Corollary~3]{Shu69a} and the 2-group
examples in~\cite[Theorem 1.1]{BG}. Thus new examples do not arise.
Sadly, when writing~\cite{BG} we were unaware of Dornhoff's
classification~\cite{Dornhoff}, so after we proved
that $\Aut(G)$ is solvable for
a finite 3-orbit 2-group $G$, we devoted considerable effort to classify~$G$.
In this paper we use Dornhoff's classification to appreciably shorten our proof.
Dornhoff's classification used the Lie ring techniques~\cite[p.\,698]{Dornhoff},
while our (lengthy) approach was elementary, just using the squaring map of $G$.
Of course, squares determine the commutators in any group $G$
because $g^{-2}(gh^{-1})^2h^2=g^{-1}h^{-1}gh=[g,h]$ for $g,h\in G$.

The proof of~\cite[Theorem~1.1]{LZ} uses an interesting technical
result~\cite[Theorem~1.2]{LZ} which we paraphrase:
Let $V=(\F_2)^n$ and let $G\le\GL(V)$ be non-solvable and act transitively on
$V\setminus\{0\}$. If there is a $G$-submodule $W$ of $\Lambda^2_{\F_2}(V)$
of co-dimension $n$, and $G$ is transitive on the non-zero
vectors of $\Lambda^2_{\F_2}(V)/W$, then the solvable residual $G^\infty$ of $G$
equals $\SL_3(2^{n/3})<\GL_n(2)$ where $3\mid n$ and $\Lambda^2_{\F_2}(V)/W$ is
the dual of the $\F_2 G^\infty$-module~$V$. Thus the techniques in~\cite{LZ}
and~\cite{BG} differ widely.

Laffey and MacHale \cite{LM86a} classified the
$3$-orbit groups that are not $p$-groups and gave structural information
for solvable $4$-orbit groups that are not $p$-groups~\cite[Theorem~2]{LM86a}.
M{\"a}urer and Stroppel gave examples of (sometimes infinite) 3-orbit groups in
\cite{MS97a}, and gave structural restrictions for such groups.
Two of the infinite families of finite 3-orbit
2-groups have infinite analogues, see~\cite[Examples 5.3]{MS97a}. They
omitted the example of order $2^9$ in~\cite[Theorem~(viii)]{Dornhoff}.
More recently (non-solvable) $k$-orbit groups with $k\in\{4,5,6\}$ have been
classified in~\cite[Theorem~3]{LM86a}, \cite[Theorem A]{BD16a}, \cite[Theorem 1]{DGB17a}.
Classifying infinite $2$-orbit
groups seems impossible; however, locally compact $k$-orbit groups,
$k\in\{2,3\}$, are classified in~\cite[Theorems 2.13 and 6.6]{Str23}.

The following result, \emph{cf.}~\cite[Proposition\;3.1]{BG}, is used to prove
Theorem~\ref{mainThm} below.

\begin{theorem}\label{normalizeTheo}
Let $G$ be a finite $2$-group such that the natural action of $\Aut(G)$ on $G$ has exactly three orbits. Then $\Aut(G)$ is solvable.
\end{theorem}

This allows us to deduce the following result from
\cite[Theorem]{Dornhoff}, which was stated in our preprint as
\cite[Theorem~1.1]{BG}. We use Higman's notations for the different
classes of Suzuki $2$-groups from \cite[Table, p.~81]{Hig63a}. To make
the formulation concise, we subsume the homocyclic groups $(C_4)^n$
and the quaternion group $\Q_8$, which are customarily not considered
to be Suzuki $2$-groups, in Higman's notations $A(n,\theta)$ and
$B(n,\theta,\varepsilon)$ by observing that $A(n,\id)\cong(C_4)^n$
(see also \cite[p.~89]{Hig63a}) and $B(1,\id,1)\cong\Q_8$, and we will
call them \enquote{improper Suzuki $2$-groups}.
Higman essentially had $n>1$.

\begin{theorem}\label{mainThm}
  Suppose that $G$ is a finite $2$-group and $\Aut(G)$ acting naturally
  on $G$ admits exactly three orbits. Then $G$ is isomorphic to one of the
  following:
\begin{enumerate}[(a)]
\item a (possibly improper) Suzuki $2$-group $A(n,\theta)$, of order $2^{2n}$, with $n\in\IN^+$, and $\theta$ an odd order automorphism of the finite field $\IF_{2^n}$, such that either $\theta=\id$, or $n$ is not a power of $2$ (in particular, $n\ge3$) and $\theta$ is nontrivial;
\item a (possibly improper) Suzuki $2$-group $B(n,\id,\mu+\mu^{-1})$, of order $2^{3n}$, with $n\in\IN^+$ and $\mu\in\IF_{2^{2n}}^{\times}$ of order $2^n+1$;
\item the (proper) Suzuki $2$-group $B(3,\theta,\varepsilon)$, of order $2^9$, where $\theta$ is any nontrivial automorphism of $\IF_{2^3}$ and
  $\varepsilon\in\{\eta,\eta^2,\eta^4\}$ where $\eta^3+\eta+1=0$.
\end{enumerate}
\end{theorem}

\subsection{Overview of the proofs of
  Theorems~\texorpdfstring{\ref{normalizeTheo}}{} and~\texorpdfstring{\ref{mainThm}}{}}\label{subsec1P3}
  
The proof of Theorem~\ref{normalizeTheo} has three steps, described
in Sections~\ref{sec2}, \ref{sec3}, \ref{sec4}.

\begin{enumerate}
\item A 3-orbit 2-group $G$ has  $\Omega_1(G)=\mho_1(G)\ne1$ where $\Omega_1(G)=\langle g\in G\mid g^2=1\rangle$ and $\mho_1(G)=\langle g^2\mid g\in G\rangle$.
  Hence the \emph{abelian} $3$-orbit $2$-groups are just the homocyclic groups $(C_4)^n$ for $n\in\IN^+$, so we will henceforth focus on finite \emph{nonabelian} $3$-orbit $2$-groups.
  Each such group $G$ has exponent $4$ with $\Z(G)=G'=\Phi(G)$ an elementary abelian $2$-group, say of $\IF_2$-dimension $n$, and $G/\Phi(G)$ another elementary abelian $2$-group, say of $\IF_2$-dimension $m$. The concatenation, of the lift to $G$ of an $\IF_2$-basis $\b_1$ of $G/\Phi(G)$, with an $\IF_2$-basis $\b_2$ of $\Phi(G)$ is a polycyclic generating sequence $\b$ of $G$ and thus associated with some (refined) consistent polycyclic presentation $P_{\b}$ of $G$ (see \cite[Section 8.1]{HEO05a} for more details). As already observed by Higman in \cite{Hig63a}, the entire presentation $P_{\b}$, and thus the group $G$, can be reconstructed from just the so-called \enquote{squaring} $\sigma_{\b}\colon\IF_2^m\rightarrow\IF_2^n$, which encodes the actual squaring function $x\mapsto x^2$ on $G$ by mapping the $\b_1$-coordinate representation of each $g\Phi(G)\in G/\Phi(G)$ to the $\b_2$-coordinate representation of the common square $g^2\in\Phi(G)$ of the elements of the coset $g\Phi(G)$. One can give conditions that characterize when a function $\sigma\colon\IF_2^m\rightarrow\IF_2^n$ is of the form $\sigma_{\b}$ for some polycyclic generating sequence $\b$ of the above described form of some nonabelian $3$-orbit group of exponent $4$. This reduces the classification of such groups to the classification of the functions $\sigma\colon\IF_2^m\rightarrow\IF_2^n$ satisfying these conditions, see Section~\ref{sec2}.

\item Let $G$ be a 3-orbit 2-groups and let $A,B$ the the groups of
  automorphisms induced on $G/\Phi(G)$ and $\Phi(G)$, respectively.
  Then $A, B$ are transitive on $\IF_2^m\setminus\{0\}$ and $\IF_2^n\setminus\{0\}$.
  One of the conditions whose conjunction characterizes when
  $\sigma\colon\IF_2^m\rightarrow\IF_2^n$ is a squaring (i.e., represents a
  nonabelian $3$-orbit $2$-group) is that there must exist
  subgroups $A\leq\GL_m(2)$ and $B\leq\GL_n(2)$ and an epimorphism
  $\varphi\colon A\twoheadrightarrow B$ where $A, B$ both act transitively
  on $\IF_2^m\setminus\{0\}$ and $\IF_2^n\setminus\{0\}$
  and such that
  $\sigma(x^g)=\sigma(x)^{\varphi(g)}$ for all $g\in A$ and all
  $x\in\IF_2^m$. The finite linear groups acting transitively on
  nonzero vectors are listed in~\cite[Appendix 1]{Lie87a}. We use this
  classification to restrict the possibilities for $A$ and $B$ and,
  ultimately, $\varphi$ and $\sigma$, as we did in our preprint
  \cite{BG}. This is done in Section~\ref{sec3}. Through a careful
  case analysis, we prove that $A$ must be solvable
  (Proposition~\ref{normalizeProp}). A simple additional argument
  shows that this implies the statement of
  Theorem~\ref{normalizeTheo}.

\item At this point, one can apply Dornhoff's classification \cite[Theorem]{Dornhoff} and conclude that the finite $3$-orbit $2$-groups are precisely the $2$-groups in Dornhoff's list. In Section~\ref{sec4}, we take a close look at those groups to see that all of them are Suzuki $2$-groups, proving Theorem~\ref{mainThm} in the process.

\end{enumerate}

\section{\texorpdfstring{$3$}{3}-orbit \texorpdfstring{$2$}{2}-groups to squarings~data and back}\label{sec2}

We will need the following facts:

\begin{proposition}\label{4presProp}
Let $m,n\in\IN^+$, and let $w_1\colon\IF_2^m\rightarrow\F(x_1,\ldots,x_m)$, $(\delta_1,\ldots,\delta_m)\mapsto x_1^{\delta_1}\cdots x_m^{\delta_m}$, and $w_2\colon\IF_2^n\rightarrow\F(y_1,\ldots,y_n)$,$(\varepsilon_1,\ldots,\varepsilon_n)\mapsto y_1^{\varepsilon_1}\cdots y_n^{\varepsilon_n}$. Moreover, let $\sigma_i\in\IF_2^n$ for $i=1,\ldots,m$ and $\pi_{i,j}\in\IF_2^n$ for $1\leq i<j\leq m$. Then:
\begin{enumerate}
\item The following power-commutator presentation $\P$ is consistent:
\begin{align*}
\P=\langle x_1,\ldots,x_m,y_1,\ldots,y_n\mid &x_i^2=w_2(\sigma_i)\text{ for }i=1,\ldots,m; \\
&[x_i,x_j]=w_2(\pi_{i,j})\text{ for }1\leq i<j\leq m; \\
&[x_i,y_k]=1\text{ for }i=1,\ldots,m\text{ and }k=1,\ldots,n; \\
&y_k^2=1\text{ for }k=1,\ldots,n;\\
&[y_k,y_l]=1\text{ for }1\leq k<l\leq n \rangle.
\end{align*}
\end{enumerate}
In what follows, the finite group represented by $\P$ is denoted by $G_{\P}$.
\begin{enumerate}\setcounter{enumi}{1}
\item Associated to $\P$, one has well-defined functions $\sigma_{\P}\colon\IF_2^m\rightarrow\IF_2^n$ with
\[
w_2(\sigma_{\P}(u))_{G_{\P}}=w_1(u)_{G_{\P}}^2
\]
for all $u\in\IF_2^m$ and $[\,,]_{\P}\colon\IF_2^m\times\IF_2^m\rightarrow\IF_2^n$ with
\[
w_2([u_1,u_2]_{\P})_G=[w_1(u_1)_G,w_1(u_2)_G]
\]
for all $u_1,u_2\in\IF_2^m$. The function $[\,,]_{\P}$ is biadditive and alternating, and
\[
  [u_1,u_2]_{\P}=\sigma_{\P}(u_1+u_2)+\sigma_{\P}(u_1)+\sigma_{\P}(u_2)
  \quad\textup{for all $u_1,u_2\in\IF_2^m$.}
\]
\item Conversely, for every pair $(\sigma;[\,,])$ such that $\sigma$ is a function $\IF_2^m\rightarrow\IF_2^n$, and $[\,,]\colon\IF_2^m\times\IF_2^m\rightarrow\IF_2^n$ is biadditive and alternating and satisfies
\[
[u_1,u_2]=\sigma(u_1+u_2)+\sigma(u_1)+\sigma(u_2)\quad\textup{for all $u_1,u_2\in\IF_2^m$,}
\]
there is a (consistent) power-commutator presentation $\P'$ as above such that $\sigma=\sigma_{\P'}$ and $[\,,]=[\,,]_{\P'}$.
\item For all $v_1,\ldots,v_m\in\IF_2^n$, the following map defined on the generators of $G_{\P}$ extends to an automorphism of $G_{\P}$: $x_i\mapsto x_iw_2(v_i)$ for $i=1,\ldots,m$, and $y_k\mapsto y_k$ for $k=1,\ldots,n$. These automorphisms form a subgroup, denoted $C_{\P}$, of the group $\Aut_c(G_{\P})$ of central automorphisms of $G_{\P}$.
\end{enumerate}
For the following final two statements, assume additionally that $[\,,]_{\P}$ (and thus also $\sigma_{\P}$) is surjective onto $\IF_2^n$.
\begin{enumerate}\setcounter{enumi}{4}
\item The map $\gamma\colon\Aut(G_{\P})\rightarrow\GL_m(2)\times\GL_n(2)$, $\alpha\mapsto(\gamma_1(\alpha),\gamma_2(\alpha))$ where $\gamma_1,\gamma_2$ are well-defined through
\begin{itemize}
\item $w_1(u^{\gamma_1(\alpha)})_{G_{\P}}\equiv w_1(v)_{G_{\P}}^{\alpha}\Mod{\langle y_1,\ldots,y_n\rangle_{G_{\P}}}$ for all $u\in\IF_2^m$, and
\item $w_2(v^{\gamma_2(\alpha)})_{G_{\P}}=w_2(v)_{G_{\P}}^{\alpha}$ for all $v\in\IF_2^n$,
\end{itemize}
is a group homomorphism with kernel $C_{\P}$ whose image is the subgroup $S_{\P}$ of $\GL_m(2)\times\GL_n(2)$ comprising the matrix pairs $(\psi_1,\psi_2)$ with $\sigma_{\P}\circ\psi_1=\psi_2\circ\sigma_{\P}$.
\item Denote by $A_{\P}$ resp.~$B_{\P}$ the projection of $S_{\P}$ onto the first resp.~second coordinate (so $A_{\P}\leq\GL_m(2)$ and $B_{\P}\leq\GL_n(2)$). Then for each $a\in A_{\P}$, there is a unique $b\in B_{\P}$ such that $(a,b)\in S_{\P}$, and the corresponding assignment $a\mapsto b$ is an epimorphism $\varphi_{\P}\colon A_{\P}\twoheadrightarrow B_{\P}$.
\end{enumerate}
\end{proposition}

These six facts are easily verified. Some were used (without proof) by Higman to classify the Suzuki $2$-groups, see \cite[beginning of Section 2, p.~80]{Hig63a}. As the authors are unaware of published proofs of facts 1--6, we will, for the sake of completeness, prove Proposition~\ref{4presProp} in the Appendix, see Subsection~\ref{app1}. Proposition~\ref{4presProp} will allow us to establish a bijective correspondence between nonabelian $3$-orbit $2$-groups and equivalence classes of squarings resp.~maximal data, defined as follows:

\begin{definition}\label{datumDef}
Let $m,n\in\IN^+$, $\sigma\colon\IF_2^m\rightarrow\IF_2^n$.

\begin{enumerate}
\item We denote by $[\,,]_{\sigma}$ the function $\IF_2^m\times\IF_2^m\rightarrow\IF_2^n$, $(x,y)\mapsto[x,y]_{\sigma}\coloneq\sigma(x+y)+\sigma(x)+\sigma(y)$.
\item An \emph{equivariance witness for $\sigma$} is an epimorphism $\varphi\colon A\twoheadrightarrow B$ with $A\leq\GL_m(2)$ acting transitively on $\IF_2^m\setminus\{0\}$, $B\leq\GL_n(2)$ acting transitively on $\IF_2^n\setminus\{0\}$, and such that $\sigma(x^g)=\sigma(x)^{\varphi(g)}$ for all $(x,g)\in \IF_2\times A$.
\item $\sigma$ is called a \emph{squaring} if and only if it has an equivariance witness and the function $[\,,]_{\sigma}$ is biadditive and nontrivial (i.e., not constantly $0$).
\item A \emph{datum} is a pair $(\sigma',\varphi)$ where $\sigma'$ is a squaring and $\varphi$ is an equivariance witness for $\sigma'$.
\item If $\delta_1\coloneq(\sigma_1,\varphi_1\colon A_1\twoheadrightarrow B_1)$ and $\delta_2\coloneq(\sigma_2,\varphi_2\colon A_2\twoheadrightarrow B_2)$ are data, say with $A_i\leq\GL_{m_i}(2)$ and $B_i\leq\GL_{n_i}(2)$ for $i=1,2$, then we say that \emph{$\delta_2$ extends $\delta_1$}, or that \emph{$\delta_1$ is a restriction of $\delta_2$}, denoted $\delta_1\leq\delta_2$, if and only if $m_1=m_2$, $n_1=n_2$, $\sigma_1=\sigma_2$, $A_1\leq A_2$ and $\varphi_1=(\varphi_2)_{\mid A_1}$. It is easy to check that this is a partial ordering on the class of data.
\item A datum is called \emph{maximal} if and only if it does not admit a proper extension.
\end{enumerate}
\end{definition}

\begin{remark}\label{datumRem}
We note the following simple facts on squarings and data:
\begin{enumerate}
\item If $\varphi\colon A\twoheadrightarrow B$ is an equivariance witness for the squaring $\sigma\colon\IF_2^m\rightarrow\IF_2^n$, then for all $g\in A$ and all $x,y\in\IF_2^m$, we have $[x^g,y^g]_{\sigma}=[x,y]_{\sigma}^{\varphi(g)}$.
\item Every squaring $\sigma$ satisfies $\sigma(0)=0$. Indeed, by the bidadditivity of $[\,,]_{\sigma}$, we have $[0,0]_{\sigma}=0$, and so $\sigma(0)=\sigma(0+0)=\sigma(0)+\sigma(0)+[0,0]_{\sigma}=0$.
\item Consequently, for every squaring $\sigma\colon\IF_2^m\rightarrow\IF_2^n$, the function $[\,,]_{\sigma}$ is alternating, i.e., $[x,x]_{\sigma}=0$ for all $x\in\IF_2^m$.
\item If $\sigma$ is a squaring, then both $\sigma$ and $[\,,]_{\sigma}$ are surjective onto $\IF_2^n$. Indeed, to prove the surjectivity of $\sigma$, it suffices to show that $\IF_2^n\setminus\{0\}\subseteq\im(\sigma)$, as we already know that $\sigma(0)=0$ by the above. Now by assumption, there is an $x\in\IF_2^m\setminus\{0\}$ with $\sigma(x)\not=0$ (otherwise, $\sigma$, and thus $[\,,]_{\sigma}$, would be constantly $0$). Also by assumption, for each given $y\in\IF_2^n\setminus\{0\}$, there is a $b\in B$ such $\sigma(x)^b=y$, and we can write $b=\varphi(a)$ for some $a\in A$. Then $\sigma(x^a)=\sigma(x)^{\varphi(a)}=\sigma(x)^b=y$. The proof of the surjectivity of $[\,,]_{\sigma}$ is analogous.
\item If $(\sigma\colon\IF_2^m\rightarrow\IF_2^n,\varphi\colon A\twoheadrightarrow B)$ is a datum and $A_1\leq A$ acts transitively on $\IF_2^m\setminus\{0\}$, then $(\sigma,\varphi_{\mid A_1})$ is also a datum.
\end{enumerate}
\end{remark}

Before introducing the notion of equivalence of data resp.~squarings, we make some simple observations. Fix $n,m\in\IN^+$, and note that the group $\GL_n(2)$ acts on the left on the class of data
\[
\delta\coloneq(\sigma\colon\IF_2^m\rightarrow\IF_2^n,\varphi\colon A\twoheadrightarrow B)
\]
via $\leftidx{^U}\delta=(\leftidx{^U}\sigma,\leftidx^{U}\varphi)$ where, for $U\in\GL_n(2)$,
\[
(\leftidx{^U}\sigma)(x)\coloneq\sigma(x)U^{-1}\text{ for all }x\in\IF_2^m,
\textup{and}\qquad
\leftidx{^U}\varphi\colon A\twoheadrightarrow B^{U^{-1}},a\mapsto\varphi(a)^{U^{-1}}.
\]
Similarly, the group $\GL_m(2)$ acts on the right on this class via $\delta^T=(\sigma^T,\varphi^T)$ where, for $T\in\GL_m(2)$,
\[
\sigma^T(x)\coloneq\sigma(xT)\text{ for all }x\in\IF_2^m,
\textup{and}\qquad
\varphi^T\colon A^T\twoheadrightarrow B,a^T\mapsto\varphi(a).
\]
These two group actions commute. We now define:

\begin{definition}\label{equivalenceDef}
Let $n,m\in\IN^+$.
\begin{enumerate}
\item Two data $\delta_1=(\sigma_1\colon\IF_2^m\rightarrow\IF_2^n,\varphi_1)$ and $\delta_2=(\sigma_2\colon\IF_2^m\rightarrow\IF_2^n,\varphi_2)$ are called \emph{equivalent} when there exist $T\in\GL_m(2)$ and $U\in\GL_n(2)$ such that $\delta_2=\leftidx{^U}\delta_1^T$.
\item Two squarings $\sigma_1,\sigma_2\colon\IF_2^m\rightarrow\IF_2^n$ are called \emph{equivalent} if and only if there are $T\in\GL_m(2)$ and $U\in\GL_n(2)$ such that $\sigma_2=\leftidx{^U}\sigma_1^T$ (equivalently, if and only if $\sigma_1$ and $\sigma_2$ are the first entries of two equivalent data).
\end{enumerate}
\end{definition}

Let us now describe functions $\Fcal_S$ resp.~$\Fcal_D$ that assign to
each (isomorphism type of a) nonabelian $3$-orbit $2$-group $G$ an
equivalence class of squarings resp.~maximal data. Since $G$ has at
most three distinct element orders, we have that $\exp(G)\mid 4$, and
actually, $\exp(G)=4$ (otherwise, $G$ is either trivial or an
elementary abelian $2$-group, so not a $3$-orbit group). Moreover, for
each $o\in\{1,2,4\}$, the subset $G_o$ of order $o$ elements in $G$ is
one of the three automorphism orbits of $G$. It follows that every
proper nontrivial characteristic subgroup of $G$ is equal to $G_1\cup
G_2$; in particular, $G'=\Z(G)=\Phi(G)=\mho^1(G)=G_1\cup
G_2$. Hence, as asserted in Subsection~\ref{subsec1P3}, both $\Phi(G)$
and $G/\Phi(G)$ are elementary abelian $2$-groups, say
$\Phi(G)\cong\IF_2^n$ and $G/\Phi(G)\cong\IF_2^m$. Also as in
Subsection~\ref{subsec1P3}, we note that each triple of choices of
$\IF_2$-bases $\b_1$ of $G/\Phi(G)$, $\b_2$ of $\Phi(G)$ and of a lift
$\tilde{\b_1}$ of $\b_1$ to $G$ yields a polycyclic generating
sequence $\b$ of $G$ in the sense of \cite[Section 2,
  p.~1446]{CEL04a}, i.e., such that $\b$ refines (also in the sense of
\cite[Section 2, p.~1446]{CEL04a}) a composition series of $G$, so the
associated consistent power-commutator presentation $\P_{\b}$ of $G$
is refined and thus (and since $\Phi(G)=\Z(G)$ is of exponent $2$)
of a form as in Proposition~\ref{4presProp}(1). Set
$\sigma_{\b}\coloneq\sigma_{\P_{\b}}$ and
$\varphi_{\b}\coloneq\varphi_{\P_{\b}}$ (using the notation from
Proposition~\ref{4presProp}(2,6) and that $[\,,]_{\P_{\b}}$ is
surjective onto $\IF_2^n$). Since both $G_2=\Phi(G)\setminus\{1\}$ and
$G_4=G\setminus\Phi(G)$ are automorphism orbits of $G$, $\varphi_{\b}$
is an equivariance witness for $\sigma_{\b}$ by
Proposition~\ref{4presProp}(5,6). Moreover, the function
$[\,,]_{\b}\coloneq[\,,]_{\P_{\b}}=[\,,]_{\sigma_{\b}}$ (in the notations of
Proposition~\ref{4presProp}(2) and Definition~\ref{datumDef}(1)) is
surjective (as already noted above, in particular nontrivial),
biadditive and alternating (by Proposition~\ref{4presProp}(2)). It
follows that $\sigma_{\b}$ is a squaring and
$(\sigma_{\b},\varphi_{\b})$ is a datum, which is maximal since
otherwise, there would exist a matrix pair
$(\psi_1,\psi_2)\in(\GL_m(2)\times\GL_n(2))\setminus\im(\gamma)$ (for
the notation $\gamma$, see Proposition~\ref{4presProp}(5)) such that
$\sigma_{\P}\circ\psi_1=\psi_2\circ\sigma_{\P}$, contradicting
Proposition~\ref{4presProp}(5). Finally, one easily checks that by the
definition of equivalence of squarings resp.~data, any other choices
of $\IF_2$-bases $\b_1'$ and $\b_2'$ of $G/\Phi(G)$ and $\Phi(G)$
respectively as well as of a lift $\tilde{\b_1'}$ of $\b_1'$ to $G$
lead to an equivalent squaring $\sigma_{\b'}$ resp.~maximal datum
$(\sigma_{\b'},\varphi_{\b'})$. This allows us to define
$\Fcal_{S}(G)$ resp.~$\Fcal_{D}(G)$ as the unique equivalence class of
squarings resp.~maximal data that are obtained through this
construction, and they depend only on the isomorphism type of $G$ as
each isomorphism $\iota\colon G_1\rightarrow G_2$, where $G_1$ (and thus
also $G_2$) is a nonabelian $3$-orbit $2$-group, maps every polycyclic
generating sequence (pcgs) $\b^{(1)}$ of $G_1$ obtained by
concatenating a lift of an $\IF_2$-basis of $G_1/\Phi(G_1)$ with an
$\IF_2$-basis of $\Phi(G_1)$ to such a pcgs $\b^{(2)}$ of $G_2$ such
that $\sigma_{\b^{(1)}}=\sigma_{\b^{(2)}}$ and
$\varphi_{\b^{(1)}}=\varphi_{\b^{(2)}}$.

\begin{proposition}\label{fProp}
The functions $\Fcal_{S}$ resp.~$\Fcal_{D}$ are bijections from the class of isomorphism types of nonabelian $3$-orbit $2$-groups to the class of equivalence classes of squarings resp.~maximal data.
\end{proposition}

\begin{proof}
It suffices to prove both surjectivity and injectivity only for $\Fcal_{S}$, as the equivalence classes of squarings are by definition in natural bijection with the equivalence classes of maximal data, and the composition of this natural bijection with $\Fcal_{S}$ is~$\Fcal_{D}$.

For surjectivity: Let $\sigma\colon\IF_2^m\rightarrow\IF_2^n$ be a squaring with equivariance witness $\varphi\colon A\twoheadrightarrow B$. By Proposition~\ref{4presProp}(3), there is a presentation $\P$ as in Proposition~\ref{4presProp}(1) such that $\sigma=\sigma_{\P}$. Let $G_{\P}$ denote the finite group represented by $\P$. By Proposition~\ref{4presProp}(5,6) and the definition of \enquote{equivariance witness}, we have $A\leq A_{\P}$, $B\leq B_{\P}$ and $\varphi=(\varphi_{\P})_{\mid A}$. Consequently, since $A$ and $B$ act transitively on $\IF_2^m\setminus\{0\}$ and $\IF_2^n\setminus\{0\}$ respectively, we conclude that $\Aut(G_{\P})$ acts transitively on the set of nontrivial cosets of $\Phi(G_{\P})=\langle y_1,\ldots,y_n\rangle_{G_{\P}}$ in $G_{\P}$, and that it also acts transitively on $\Phi(G_{\P})\setminus\{1\}$. So in order to show that the action of $\Aut(G_{\P})$ on $G_{\P}$ admits just the three orbits $\{1\}$, $\Phi(G_{\P})\setminus\{1\}$ and $G_{\P}\setminus\Phi(G_{\P})$, it only remains to check that the coset $(x_1)_{G_{\P}}\Phi(G_{\P})$ is contained in a single orbit, which is clear from Proposition~\ref{4presProp}(4).

For injectivity: Let $G_1$ and $G_2$ be nonabelian $3$-orbit $2$-groups that are mapped by $\Fcal_{S}$ to equivalent squarings $\sigma^{(1)}$ and $\sigma^{(2)}$, say both mapping $\IF_2^m\rightarrow\IF_2^n$. Then $G_k$, $k\in\{1,2\}$, is represented by the consistent power-commutator presentation $\P_k$ which is of the same form as the presentation $\P$ in Proposition~\ref{4presProp}(1), but with respect to the choices $\sigma_i\coloneq\sigma_i^{(k)}\coloneq\sigma^{(k)}(e_i)$ for $i=1,\ldots,m$ and $\pi_{i,j}\coloneq\pi_{i,j}^{(k)}\coloneq\sigma^{(k)}(e_i+e_j)+\sigma^{(k)}(e_i)+\sigma^{(k)}(e_j)$ for $1\leq i<j\leq m$. Identify $G_1$ with the quotient $\F(x_1,\ldots,x_m,y_1,\ldots,y_n)/R$, where $R$ is the normal subgroup of the free group $\F(x_1,\ldots,x_n,y_1,\ldots,y_m)$ corresponding to the defining relations of $\P_1$. Let $\tilde{\b_1}$ be the initial segment consisting of those entries that are outside $\Phi(G_1)$ of the standard pcgs $\b$ of $G_1$ with respect to the presentation $\P_1$, let $\b_1$ be the projection of $\tilde{\b_1}$ into $G_1/\Phi(G_1)$, and let $\b_2$ be the terminal segment of $\b$ consisting of those entries that are inside $\Phi(G_1)$. Then, writing $\sigma^{(2)}=\leftidx{^U}(\sigma^{(1)})^T$ for $U\in\GL_n(2)$ and $T\in\GL_m(2)$, it is immediate to check that the consistent power-commutator presentation of $G_1$ associated with the concatenation of $\widetilde{\b_1T}$ (an arbitrary lift of $\b_1T$ to $G_1$) and $\b_2U^{-1}$ is $\P_2$, and so $G_1\cong G_2$, as required.
\end{proof}

\section{Excluding insolvable data in Theorem~\ref{normalizeTheo}}\label{sec3}

In the spirit of the previous section, in order to classify the nonabelian $3$-orbit $2$-groups, it suffices to do so for the equivalence classes of squarings resp.~(maximal) data. We make essential use of the known classification of finite linear groups that act transitively on nonzero vectors, due to Huppert \cite{Hup57a} and Hering \cite{Her85a} and customarily called Hering's theorem, see e.g.~\cite[Appendix 1]{Lie87a}. We note that each of the symbols $\GammaL_1(p^d)$, $\SL_a(q)$, $\Sp_{2a}(q)$ and $G_2(q)'$ used in the formulation of Hering's theorem from \cite{Lie87a} denotes a subgroup of the corresponding general linear group $\GL_d(p)$ that is unique up to $\GL_d(p)$-conjugacy (and thus up to permutation group isomorphism); in particular, $G_2(q)'$ there denotes any (necessarily maximal and unique up to conjugacy by \cite[tables 8.28 and 8.29, pp.~391f.]{BHR13a}) copy of that group in (one of the copies of) $\Sp_6(q)\leq\GL_6(q)\leq\GL_d(p)$ with $q^6=p^d$, see also \cite[proof of Hering's theorem, p.~513]{Lie87a}.

Let $(\sigma\colon\IF_2^m\rightarrow\IF_2^n,\varphi\colon A\twoheadrightarrow B)$ be a datum. By Hering's theorem, one of the following three cases applies to $A\leq\GL_m(2)$:
\begin{enumerate}
\item $A\unrhd\SL_a(2^b)$ with $m=a\cdot b$ and $a>1$, or $A\unrhd\Sp_a(2^b)$ with $m=a\cdot b$ and $2\mid a$, or $A\unrhd G_2(2^b)'$ with $m=6b$.
\item $A\leq\GammaL_1(2^m)$.
\item $m=4$, and $A\in\{\Alt(6),\Alt(7)\}$.
\end{enumerate}

Analogously, $B\leq\GL_n(2)$ satisfies one of the following:
\begin{enumerate}
\item $B\unrhd\SL_c(2^d)$ with $n=c\cdot d$ and $c>1$, or $B\unrhd\Sp_c(2^d)$ with $n=c\cdot d$ and $2\mid c$, or $B\unrhd G_2(2^d)'$ with $n=6d$.
\item $B\leq\GammaL_1(2^n)$.
\item $n=4$, and $B\in\{\Alt(6),\Alt(7)\}$.
\end{enumerate}

Note that the third case is impossible for $B$ if $m\geq 5$, since none of $\Alt(6)$ and $\Alt(7)$ is a composition factor of $A$ then. Also, note that $A$ resp.~$B$ is solvable if and only if it is contained in $\GammaL_1(2^m)$ resp.~$\GammaL_1(2^n)$ (using that $\SL_2(2)=\Sp_2(2)=\GammaL_1(2^2)$ equals the full general linear group $\GL_2(2)$).

We will now show that $A$ and thus also $B$ must actually be solvable.

\begin{proposition}\label{normalizeProp}
There are no data $(\sigma\colon\IF_2^m\rightarrow\IF_2^n,\varphi\colon A\twoheadrightarrow B)$ where $A$ is insolvable, i.e., such that one of the following holds:
\begin{enumerate}
\item $m\geq3$ and $A\unrhd\SL_a(2^b)$ for some $a,b\in\IN^+$ with $m=a\cdot b$ and $a>1$.
\item $m\geq4$ and $A\unrhd\Sp_a(2^b)$ for some $a,b\in\IN^+$ with $m=a\cdot b$ and $2\mid a$.
\item $m\geq6$ and $A\unrhd G_2(2^b)'$ with $b=m/6$.
\end{enumerate}
\end{proposition}

\begin{proof}
Assume, for a contradiction, that such a datum exists. We give a detailed argument for the case $A\unrhd\SL_a(2^b)$ and then explain how basically the same argument can also be applied in the other two cases.

By Remark~\ref{datumRem}(5), we may assume w.l.o.g.~that $A=\SL_a(2^b)$. In particular, $\PSL_a(2^b)$ is the only nonabelian composition factor of $A$. Since $B$ is a quotient of $A$ acting transitively on $\IF_2^n\setminus\{0\}$, this leaves only the following possibilities for $B$:
\begin{enumerate}
\item $n=m$, $B\unrhd\SL_a(2^b)$ with respect to a possibly different $\IF_{2^b}$-vector space structure on $\IF_2^m$.
\item $B\leq\GammaL_1(2^n)$.
\end{enumerate}

In the latter case, $B$ is solvable and therefore (since $\SL_a(2^b)$ is perfect), $\varphi$ is trivial. But since $A=\SL_a(2^b)$ acts transitively on $\IF_{2^b}^a\setminus\{0\}$, this implies that $\sigma$ is constant on $\IF_{2^b}^a\setminus\{0\}$, say $\sigma(x)=c$ for all $x\in\IF_{2^b}^a\setminus\{0\}$. It follows that for any two distinct nonzero vectors $x,y\in\IF_{2^b}^a$,
\[
[x,y]=\sigma(x+y)+\sigma(x)+\sigma(y)=c+c+c=c.
\]
In particular, for any fixed nonzero vector $x\in\IF_{2^b}^a$, we have $[x,y]=c$ for all but at most two $y\in\IF_{2^b}^a$. Since $ab=m\geq3$ by assumption and $[\,,]$ is biadditive, this is only possible if $c=0$, which implies that $\sigma$ and thus $[\,,]$ is trivial, a contradiction.

So we may assume that the former case for $B$ applies. Then $|\SL_a(2^b)|=|A|\geq|B|\geq|\SL_a(2^b)|$, so $B$ itself is some copy of $\SL_a(2^b)$ in $\GL_m(2)$, possibly a different one than $A$. We will continue to write $\IF_2^m=\IF_{2^b}^a$ with respect to the $\IF_{2^b}$-vector space structure associated with $A=\SL_a(2^b)$ (and whenever we write $\SL_a(2^b)$, we mean $A$). There exists $\gamma\in\GL_n(2)$ such that $B=A^{\gamma}=A^{\conj(\gamma)}$. It follows that $\varphi\conj(\gamma)^{-1}$ is an automorphism of $A$. Since $\SL_a(2^b)$ is quasisimple with central quotient $\PSL_a(2^b)$, we have $\Out(\SL_a(2^b))\cong\Out(\PSL_a(2^b))$, and so $\varphi\conj(\gamma)^{-1}$, like any automorphism of $\SL_a(2^b)$, can be written in the form $\iota^{\varepsilon}\conj(\beta)$ where $\beta$ is  an element of the semilinear group $\GammaL_a(2^b)$, $\iota$ is the inverse-transpose automorphism and $\varepsilon\in\{0,1\}$. Hence $\varphi=(\iota^{\varepsilon}\conj(\beta\gamma))_{\mid A}=(\iota^{\varepsilon}\conj(\beta'))_{\mid\SL_a(2^b)}$, where $\beta'\coloneq\beta\gamma$ is a bijective additive map on $\IF_{2^b}^a$.

We now show that we may assume that $\varepsilon=0$. Indeed, if $a=2$, then this is clear since $\iota$ is an inner automorphism of $\GL_a(2^b)=\GL_2(2^b)$ then, so assume $a\geq3$. If $\varepsilon=1$, then since for all $g\in\SL_a(2^b)$ and all $x\in\IF_{2^b}^a$, $\sigma(x)^{\varphi(g)}=\sigma(x^g)$, we find that the elements of $\Stab_{\SL_a(2^b)}(x)^{\varphi}=\Stab_{\SL_a(2^b)}(x)^{\iota\beta'}$ have $\sigma(x)$ as a common fixed point, and so the elements of $\Stab_{\SL_a(2^b)}(x)^{\iota}=\{g^{\tr}\mid g\in\Stab_{\SL_a(2^b)}(x)\}$ have $\sigma(x)^{(\beta')^{-1}}$ as a common fixed point. But choosing $x\coloneq(1,0,\ldots,0)$, we find that the elements of $\Stab_{\SL_a(2^b)}(x)^{\iota}$ are just the $a\times a$ matrices over $\IF_{2^b}$ of the form
\[
\scalemath{0.9}{\begin{pmatrix}
1 & b_1 & b_2 & \cdots & b_{a_1} \\
0 & c_{1,1} & c_{1,2} & \cdots & c_{1,a-1} \\
0 & c_{2,1} & c_{2,2} & \cdots & c_{2,a-1} \\
\vdots & \vdots & \vdots & \vdots & \vdots \\
0 & c_{a-1,1} & c_{a-1,2} & \cdots & c_{a-1,a-1}
\end{pmatrix}}
\]
where $b_1,\ldots,b_{a-1}$ are arbitrary elements of $\IF_{2^b}$ and $(c_{i,j})_{i,j=1}^{a-1}$ is an arbitrary element of $\SL_{a-1}(2^b)$. We argue that for every nonzero vector $y=(y_1,\ldots,y_a)$ in $\IF_{2^b}^a$, there is a matrix $M(y)$ of the above form that does not fix $y$, yielding a contradiction. Indeed, if $y_1$ is the only nonzero entry of $y$, then just let $M(y)$ be any matrix of the above form where $b_1\not=0$. And if $(y_2,\ldots,y_a)\not=(0,\ldots,0)$, then choose $b_1=b_2=\cdots=b_{a-1}=0$ and let $(c_{i,j})_{i,j=1}^{a-1}\in\SL_{a-1}(2^b)$ be such that it does not fix $(y_2,\ldots,y_a)$ (here, we need to use that $a\geq3$).

So we may indeed assume $\varepsilon=0$, and thus $\varphi=\conj(\beta')_{\mid\SL_a(2^b)}$ for some bijective, additive map $\beta'$ on $\IF_{2^b}^a$. Using once more that for all $g\in\SL_a(2^b)$ and all $x\in\IF_{2^b}^a$,
\[
\sigma(x)^{g^{\beta'}}=\sigma(x)^{\varphi(g)}=\sigma(x^g),
\]
we can now conclude that
\begin{align*}
  \Stab_{\SL_a(2^b)}(x)\subseteq\varphi^{-1}[\Stab_{\varphi[\SL_a(2^b)]}(\sigma(x))]&=\Stab_{\SL_a(2^b)^{\beta'}}(\sigma(x))^{(\beta')^{-1}}\\&=\Stab_{\SL_a(2^b)}(\sigma(x)^{(\beta')^{-1}}),
\end{align*}
and this inclusion must be an equality.
Hence
\[
\IF_{2^b}x=\fix_{\IF_{2^b}^a}(\Stab_{\SL_a(2^b)}(x))=\fix_{\IF_{2^b}^a}(\Stab_{\SL_a(2^b)}(\sigma(x)^{(\beta')^{-1}}))=\IF_{2^b}\sigma(x)^{(\beta')^{-1}}
\]
for all $x\in\IF_{2^b}^a\setminus\{0\}$. In particular, there exists a function $\lambda\colon\IF_{2^b}^a\rightarrow\IF_{2^b}$ such that for all $x\in\IF_{2^b}^a$, $\sigma(x)^{(\beta')^{-1}}=\lambda(x)x$, or equivalently, $\sigma(x)=(\lambda(x)x)^{\beta'}$, and we may assume w.l.o.g.~that $\lambda(0)=\lambda(v_0)$ for any fixed $v_0\in\IF_{2^b}^a\setminus\{0\}$. Moreover, since $[x,y]=\sigma(x+y)+\sigma(x)+\sigma(y)$ for all $x,y\in\IF_{2^b}^a$, we get that

\begin{equation}\label{eq7}
[x,y]=((\lambda(x+y)+\lambda(x))x+(\lambda(x+y)+\lambda(y))y)^{\beta'}.
\end{equation}

Now for all $g\in\SL_a(2^b)$ and all $x\in\IF_{2^b}^a$, we have that
\[
((\lambda(x)x)^g)^{\beta'}=((\lambda(x)x)^{\beta'})^{g^{\beta'}}=\sigma(x)^{g^{\beta'}}=\sigma(x^g)=(\lambda(x^g)x^g)^{\beta'},
\]
so that $\lambda(x^g)=\lambda(x)$. As $\SL_a(2^b)$ acts transitively on the nonzero vectors, this entails that $\lambda$ is constant on $\IF_{2^b}^a\setminus\{0\}$ and thus constant as a whole. Hence by Formula (\ref{eq7}), $[\,,]$ is trivial, a contradiction.

Let us now show why this argument can also be applied in the other two cases. The crucial observation is that we have actually only used a few properties of the group $\SL_a(2^b)$ and of its action on $\IF_{2^a}^b$ in the above argument, namely:
\begin{enumerate}
\item Every group automorphism of $\SL_a(2^b)\leq\GL_a(2^b)$ can be written as $\iota^{\varepsilon}\conj(\beta)$, where $\beta$ is a suitable bijective, additive map on $\IF_{2^b}^a$, $\iota$ is the inverse-transpose automorphism and $\varepsilon\in\{0,1\}$. The elements of the image $\Stab_{\SL_a(2^b)}((1,0,\ldots,0))^{\iota}$ do not have a common nonzero fixed point.
\item $\SL_a(2^b)$ acts transitively on $\IF_{2^b}^a\setminus\{0\}$ (a well-known and easy to check fact, which also conveniently follows from the \enquote{conversely} clause in \cite[Hering's theorem, Appendix 1]{Lie87a}).
\item $\PSL_a(2^b)$ is the only nonabelian composition factor of $\SL_a(2^b)$.
\item $\SL_a(2^b)$ is perfect.
\item For each $x\in\IF_{2^b}^a\setminus\{0\}$, we have that $\fix_{\IF_{2^b}^a}(\Stab_{\SL_a(2^b)}(x))=\IF_{2^b}x$.
\end{enumerate}

Analogous properties also hold for the groups with which the other two cases are concerned, so that the argument can indeed be transferred. More precisely:
\begin{enumerate}
\item If $a\not=4$, then every group automorphism of $\Sp_a(2^b)\leq\GL_a(2^b)$ can be written as $\conj(\beta)$, where $\beta$ is a suitable element of $\GammaSp_a(2^b)$; in particular, $\beta$ is additive. Furthermore, every group automorphism of $\Sp_4(2^b)\leq\GL_4(2^b)$ can be written as $\iota^{\varepsilon}\conj(\beta)$ where $\beta$ is a suitable element of $\GammaSp_4(2^b)$, $\iota$ is the explicit graph automorphism of $\Sp_4(2^b)$ described in \cite[p.~235]{Sol77a} and $\varepsilon\in\{0,1\}$. With this choice of $\iota$, some easy computations (see \S\S\ref{app2} of the Appendix for details) show that the elements of $\Stab_{\Sp_4(2^b)}((1,0,\ldots,0))^{\iota}$ do not have a common nonzero fixed point.
\item $\Sp_a(2^b)$ acts transitively on $\IF_{2^b}\setminus\{0\}$ by \cite[Hering's theorem, Appendix 1]{Lie87a}.
\item $\PSp_a(2^b)=\Sp_a(2^b)$ is the only nonabelian composition factor of $\Sp_a(2^b)$.
\item  $\Sp_a(2^b)$ is perfect.
\item For each $x\in\IF_{2^b}^a\setminus\{0\}$, we have that $\fix_{\IF_{2^b}^a}(\Stab_{\Sp_a(2^b)}(x))=\IF_{2^b}x$, which we will now prove. We need to show that for each $y\in\IF_{2^b}^a\setminus\IF_{2^b}x$, there is a $g\in\Stab_{\Sp_a(2^b)}(x)$ such that $y^g\not=y$. Denote by $\langle,\rangle$ a symplectic form on $\IF_{2^b}^a$ preserved by $\Sp_a(2^b)$. Using that $\Sp_a(2^b)$ acts transitively on the set of pairs $(v,w)\in(\IF_{2^b}^a)^2$ with $\IF_{2^b}$-linearly independent entries such that $\langle v,w\rangle=\langle x,y\rangle$ (which holds by an application of Witt's lemma), we conclude that there exists $g\in\Sp_a(2^b)$ such that $x^g=x$ and $y^g=x+y\not=y$, as required.
\end{enumerate}

Similar arguments can be applied to the groups $G_2(2^b)'$:
\begin{enumerate}
\item Every group automorphism of $G_2(2^b)'\leq\GL_6(2^b)$ can be written as $\conj(\beta)$, where $\beta$ is a suitable bijective, additive map on $\IF_{2^b}^6$. Indeed, for $b=1$, this follows from the fact that $G_2(2)\leq\GL_6(2)$ acts faithfully by conjugation on $G_2(2)'$ (and thus achieves the full automorphism group $\Aut(G_2(2)')$), see \cite[page on $G_2(2)'$]{ATLAS}. For $b\geq 2$, where $G_2(2^b)'=G_2(2^b)$, one can see this as follows: As noted at the beginning of this section, $G_2(2^b)$ may be seen as a member of a unique conjugacy class of maximal subgroups of $\Sp_6(2^b)$. Let $\sigma$ be a generator of the field automorphism subgroup of $\GammaSp_6(2^b)$, acting on $\Sp_6(2^b)$. Then $G_2(2^b)^{\sigma}=G_2(2^b)^g$, or equivalently $G_2(2^b)^{\sigma\conj(g)^{-1}}=G_2(2^b)$, for some $g\in\Sp_6(2^b)$. Now if we can argue that for $k=1,\ldots,b-1$, the action on $G_2(2^b)$ of the power $\psi_k\coloneq(\sigma\conj(g)^{-1})^k$ is \emph{not} by some inner automorphism of $G_2(2^b)$, it will follow (from the known order of $\Aut(G_2(2^b))$) that every automorphism of $G_2(2^b)$ can be written as a composition of an inner automorphism of $G_2(2^b)$ with some power of $\sigma\conj(g)^{-1}$; in particular, our assertion follows. So assume, aiming for a contradiction, that for some $k\in\{1,\ldots,b-1\}$ and some $h\in G_2(2^b)$, $\psi_k\conj(h)$ centralizes $G_2(2^b)$. Viewing $\psi_k\conj(h)$ as an automorphism of $\Sp_6(2^b)$ which involves a nontrivial field automorphism, we get by \cite[Proposition~4.1(I)]{Har92a} that the order of the centralizer of $\psi_k\conj(h)$ in $\Sp_6(2^b)$ cannot exceed $|\Sp_6(2^c)|$ where $c$ is the largest proper divisor of $b$. But $|\Sp_6(2^c)|=2^{9c}(2^{2c}-1)(2^{4c}-1)(2^{6c}-1)\leq 2^{9b/2}(2^b-1)(2^{2b}-1)(2^{3b}-1)$, which is less than $|G_2(2^b)|=2^{6b}(2^{6b}-1)(2^{2b}-1)$, yielding the desired contradiction.
\item $G_2(2^b)'$ acts transitively on $\IF_{2^b}^6\setminus\{0\}$ by~\cite[Hering's theorem, Appendix 1]{Lie87a}.
\item $G_2(2^b)'$ is the only nonabelian composition factor of $G_2(2^b)'$.
\item $G_2(2^b)'$ is perfect.
\item For each $x\in\IF_{2^b}^6\setminus\{0\}$, we have that $\fix_{\IF_{2^b}^6}(\Stab_{G_2(2^b)'}(x))=x\IF_{2^b}$. For $b=1$, this can be checked with GAP \cite{GAP4}, using the two matrix generators of $G_2(2)'\leq\GL_6(2)$ from the ATLAS of Finite Group Representations \cite[page on $G_2(2)'$]{ATLAS}. So we may assume that $b\geq2$, whence $G_2(2^b)'=G_2(2^b)$. Note that by \cite[Lemma 5.1]{Coo81a}, the action of $G_2(2^b)$ on the lines of $\IF_{2^b}^6$ is primitive. Moreover, also by \cite[Lemma 5.1]{Coo81a}, the action of $G_2(2^b)$ on $\IF_{2^b}^6\setminus\{0\}$ cannot be regular, as $\Stab_{G_2(2^b)}(\IF_{2^b}x)$ contains a nontrivial $2$-subgroup, the elements of which must fix $\IF_{2^b}x$ pointwise. It follows that setting $S_y\coloneq\fix_{\IF_{2^b}^6}(\Stab_{G_2(2^b)}(y))\supseteq\IF_{2^b}y$ for each $y\in\IF_{2^b}^6\setminus\{0\}$, the $S_y$ all are proper nontrivial subspaces of $\IF_{2^b}^6$ of a common $\IF_{2^b}$-dimension (the latter by the transitivity of $G_2(2^b)$ on nonzero vectors). Moreover, the set $\{S_y\mid y\in\IF_{2^b}^6\setminus\{0\}\}$ is a $G_2(2^b)$-invariant partition of $\IF_{2^b}^6\setminus\{0\}$ corresponding to a $G_2(2^b)$-invariant partition of the set of lines of $\IF_{2^b}^6$. So by primitivity, each $S_y$ must be a line, whence $S_y=\IF_{2^b}y$ for all $y\in\IF_{2^b}\setminus\{0\}$. In particular, we get the required equality $S_x=\IF_{2^b}x$.\qedhere
\end{enumerate}
\end{proof}

Using Proposition~\ref{normalizeProp}, we can prove Theorem~\ref{normalizeTheo}.

\begin{proof}[Proof of Theorem~\ref{normalizeTheo}]
For every nonabelian $3$-orbit $2$-group $G$, the induced action of $\Aut(G)$ on the Frattini quotient $G/\Phi(G)$ has kernel $\Aut_c(G)$, a $2$-group (see Proposition~\ref{4presProp}(5) and use that $\Phi(G)=\Z(G)$ here). In particular, $\Aut_c(G)$ is solvable. Moreover, the induced action of $\Aut(G)$ on $G/\Phi(G)$ is solvable by Proposition~\ref{normalizeProp}. Hence, $\Aut(G)$ is an extension of two solvable groups, whence it is solvable itself.
\end{proof}

\section{Proof of Theorem~\ref{mainThm} and isomorphisms}\label{sec4}

By Theorem~\ref{normalizeTheo}, we know that the automorphism group of each finite $3$-orbit $2$-group $G$ is solvable. Therefore, these groups are precisely the non-abelian $2$-groups of Dornhoff's classification in \cite[Theorem]{Dornhoff}, which are the following:
\begin{enumerate}
\item the homocyclic groups $(C_4)^n$ for $n\in\IN^+$, which may be viewed as the improper (see the paragraph before Theorem~\ref{mainThm}) Suzuki $2$-groups $A(n,\id)$;
\item the proper Suzuki $2$-groups $A(n,\theta)$ where $n\geq3$ is not a power of $2$ and $\theta\in\Aut(\IF_{2^n})$ is a nontrivial of odd order;
\item the groups $B(n)=\IF_{2^{2n}}\times\IF_{2^n}$ with $n\geq1$,
  $\mu\in\IF_{2^{2n}}^\times$ of order $2^n+1$, and multiplication rule
\[
(\alpha,\zeta)\cdot(\beta,\eta)=(\alpha+\beta,\zeta+\eta+\alpha\beta^{2^n}\mu+\alpha^{2^n}\beta\mu^{-1}).
\]
We allow $n=1$ to include the improper Suzuki $2$-group $B(1,\id,1)\cong\Q_8$;
\item the groups $P(\varepsilon)=\IF_{2^6}\times\IF_{2^3}$, where
  $\varepsilon\in\IF_{2^6}$ has order 63, with multiplication
\[
 (\alpha,\zeta)\cdot(\beta,\eta)=(\alpha+\beta,\zeta+\eta+\alpha\beta^2\varepsilon+\alpha^8\beta^{16}\varepsilon^8).
\]
\end{enumerate}

If $\eta\in\IF_8$ satisfies $\eta^3+\eta+1=0$, then
$\IF_8^\times=\langle\eta\rangle$ and
\[
\{\eta,\eta^2,\eta^4\}
=\IF_8\setminus\{\rho^{-1}+\rho^2\mid \rho\in\IF_8^\times\}
=\IF_8\setminus\{\rho^{-1}+\rho^4\mid \rho\in\IF_8^\times\}.
\]
Hence the group $B(3,\theta,\varepsilon)$ in Theorem~\ref{mainThm}(c)
is a Suzuki $2$-group as defined by Higman. It was proved
in~\cite[Remark, p.~706]{Dornhoff} that his group $B(n)$
in Theorem~1.1(v) is isomorphic to Higman's group $B(n,1,\mu+\mu^{-1})$ (note that, as required
by Higman, $\mu+\mu^{-1}$ is not of the form $\rho+\rho^{-1}$ for any $\rho\in\IF_{2^n}$),
and in~\cite[Lemma~4.6]{LZ} that $P(\varepsilon)\cong P(\varepsilon')$ for all primitive elements
$\varepsilon,\varepsilon'\in\IF_{2^6}$.  In order to complete the
proof of Theorem~\ref{mainThm}, we need to argue that the group
$P(\varepsilon)$ defined by Dornhoff~\cite[Theorem~(viii)]{Dornhoff}
is isomorphic to each of the Suzuki $2$-groups $B(3,\theta,\varepsilon)$ of
Theorem~\ref{mainThm}(c): all of these groups have order $2^9$. We also must
show that $P(\varepsilon)\not\cong B(3)=B(3,1,\mu+\mu^{-1})$. These calculations
can be performed by a computer. However, the latter calculation is very
delicate: the command {\tt IsIsomorphic} in {\sc Magma} crashed our computer,
also $P(\varepsilon)$ and $B(3)$ share a very large number of
isomorphism invariants, but we did find that they had a different number of
subgroups isomorphic to $(C_4)^3$: either zero or nine.

Given the great interest in proving isomorphism and non-isomorphism in
finite $p$-groups, we shall give non-computational proofs, first
that $P(\varepsilon)\cong B(3,\theta,\varepsilon)$ and second
that $P(\varepsilon)\not\cong B(3)$. Let $\omega\in\IF_{2^6}$ be any root of the
irreducible polynomial $X^6+X^4+X^3+X+1$, and note
that $\IF_{2^6}^\times=\langle\omega\rangle$.

In order to proceed, let us describe the unique group of the form $P(\varepsilon)$ in terms of a squaring. We note that in $P(\omega)$, one has
\[
(\alpha,0)\cdot(\alpha,0)=(0,\omega\alpha^3+\omega^8\alpha^{24}),
\]
whence the function $\sigma\colon\IF_{2^6}\rightarrow\IF_{2^3},\alpha\mapsto\omega\alpha^3+\omega^8\alpha^{24}$, may be viewed as a squaring of $P(\omega)$. Note that for this, we identify the fields $\IF_{2^6}$ and $\IF_{2^3}$ with the $\IF_2$-vector spaces $\IF_2^6$ and $\IF_2^3$ by fixing $\IF_2$-bases in the two fields.

Continuing to work with those two field structures, we can also provide an explicit equivariance witness for $\sigma$, which requires some notation. Note that $\GammaL_1(2^6)=\Gal(\IF_{2^6}/\IF_2)\ltimes\IF_{2^6}^\times$ and $\GammaL_1(2^3)=\Gal(\IF_{2^3}/\IF_2)\ltimes\IF_{2^3}^\times$. We write $\Gal(\IF_{2^6}/\IF_2)=\langle\alpha\rangle$ and $\Gal(\IF_{2^3}/\IF_2)=\langle\alpha^2\rangle$. Moreover, denoting by $\hat{\omega}$ the multiplication by $\omega$ on $\IF_{2^6}$, the subgroup $\IF_{2^6}^\times$ of $\GammaL_1(2^6)$ is generated by $\hat{\omega}$. Viewing $\IF_{2^3}^\times$ as a subgroup of $\IF_{2^6}^\times$, the corresponding subgroup of $\GammaL_1(2^3)$ is generated by (the restriction of) $\hat{\omega}^9$.

Then the aforementioned equivariance witness for $\sigma$ is the epimorphism
\[
\varphi\colon\langle\alpha^2\hat{\omega},\hat{\omega}^3\rangle\rightarrow\GammaL_1(2^3),\ \ \alpha^2\hat{\omega}\mapsto\alpha^4,\hat{\omega}^3\mapsto\hat{\omega}^9.
\]
The reader may like to check that $\varphi$ is indeed an equivariance witness for $\sigma$.

It is also not hard to check that the group $P(\omega)$ is nonabelian. According to the above equivariance witness, the restriction of the automorphism group of $P(\omega)$ to $\Phi(P(\omega))$ contains a copy of $\GammaL_1(2^3)$. In particular, it contains a Singer cycle, whence $P(\omega)$ is a (proper) Suzuki $2$-group. We now aim to prove that it has any of the \enquote{Higman notations} that are indicated in part (3) of our Theorem~\ref{mainThm}.

First, note that there are no $A$-type and $D$-type Suzuki $2$-groups of order $2^9$. Moreover, for any fixed automorphism $\xi$ of $P(\omega)$ that induces the order $7$ automorphism $\hat{\omega}^9\in\GammaL_1(2^6)$ on $P(\omega)/\Phi(P(\omega))$, the group $P(\omega)$ has exactly $9$ normal subgroups of order $2^6$ that are $\xi$-invariant, and for any concrete choice of $\xi$, one can easily check with GAP that these subgroups are all nonabelian. Hence by reversing the argument in \cite[proof of Lemma 12, pp.~90--92]{Hig63a}, we find that $P(\omega)$ also cannot be a $C$-type Suzuki $2$-group, whence it is a $B$-type Suzuki $2$-group.

Our next goal is to argue that $P(\omega)\not\cong B(3,\id,\varepsilon)$ for any $\varepsilon\in\IF_{2^3}^\times$ not of the form $\rho^{-1}+\rho$ for some $\rho\in\IF_{2^3}^\times$. For this, we need some preparations.

We exhibit a common datum for the groups $B(3,\id,\varepsilon)$, which shows in particular that they are pairwise isomorphic. The squaring for this datum is the function $\sigma'\colon\IF_{2^6}\rightarrow\IF_{2^3},x\mapsto x^9$, and its equivariance witness is the epimorphism $\varphi'\colon\langle\hat{\omega}\rangle\rightarrow\langle\hat{\omega}^9\rangle, \hat{\omega}\mapsto\hat{\omega}^9$. This squaring represents some $3$-orbit $2$-group with elements of the form $(\gamma,\zeta)$ where $\gamma\in\IF_{2^6}$ and $\zeta\in\IF_{2^3}$. On the other hand, Higman writes the groups of the form $B(3,\id,\varepsilon)$ such that the elements are triples $(\alpha,\beta,\zeta)$ where $\alpha,\beta,\zeta\in\IF_{2^3}$. Let $P(X)=X^2+\varepsilon X+1\in\IF_{2^3}[X]$ and identify $\IF_{2^6}$ with $\IF_{2^3}[X]/(P(X))$, so that the elements of $\IF_{2^6}$ have the form $\xi_0+\xi_1X+(P(X))$ with $\xi_0,\xi_1\in\IF_{2^3}$. The squaring $\sigma$ is the $(2^3+1)^{\rm th}$ power function $\IF_{2^6}\rightarrow\IF_{2^3}$. Computing modulo $P(X)$, gives
\begin{align*}
(\xi_0+\xi_1X)^{1+2^3}&\equiv(\xi_0+\xi_1X)\cdot(\xi_0+\xi_1X)^{2^3}
\equiv(\xi_0+\xi_1X)\cdot(\xi_0+\xi_1X^{2^3})\\
&\equiv(\xi_0+\xi_1X)\cdot(\xi_0+\xi_1(\varepsilon+X))=
(\xi_0+\xi_1X)\cdot(\xi_0+\xi_1\varepsilon+\xi_1X)\\
&\equiv\xi_0^2+\varepsilon\xi_0\xi_1+X\cdot(\xi_0\xi_1+\xi_1\xi_0+\xi_1^2\varepsilon)+X^2\cdot\xi_1^2\\
&\equiv \xi_0^2+\varepsilon\xi_0\xi_1+X\cdot\xi_1^2\varepsilon+(1+\varepsilon X)\cdot\xi_1^2\equiv \xi_0^2+\varepsilon\xi_0\xi_1+\xi_1^2.
\end{align*}
This is the same as the squaring function in $B(3,\id,\varepsilon)$, and the only constraint on $\varepsilon$ here is that the polynomial $X^2+\varepsilon X+1$ is irreducible over $\IF_{2^3}$, which is equivalent to Higman's condition that $\varepsilon\not=\rho^{-1}+\rho$ for all $\rho\in\IF_{2^3}$. This concludes the proof of our claim that all Suzuki $2$-groups of the form $B(3,\id,\varepsilon)$ are isomorphic to the $3$-orbit $2$-group represented by the above datum $(\sigma',\varphi')$.

Now we can argue that $P(\omega)$ is \emph{not} represented by this datum, whence it does not appear in Theorem~\ref{mainThm}(b). Assume otherwise. Then it is possible to extend the given equivariance witness $\varphi\colon A\rightarrow\GammaL_1(2^3)$ for the squaring $\sigma$ representing $P(\omega)$ such that the extended domain $A^+$ of the extended equivariance witness $\varphi^+$ contains an $\IF_2$-linear vector space automorphism of $\IF_{2^6}$ of order $2^6-1$. Since $A^+$ must be contained in some conjugate of $\GammaL_1(2^6)$, it follows that the order $21$ element $\hat{\omega}^3\in A\leq A^+$ must be a third power of a generator $\gamma$ of the unique cyclic subgroup of $A^+$ of order $63$. In particular, $\gamma$ lies in the centralizer of $\hat{\omega}^3$ in $\GL_6(2)$, which by \cite[Lemma 5.3]{AW92a} implies that $\langle\gamma\rangle=\langle\hat{\omega}\rangle$. Thus, $\hat{\omega}\in A^+$. Moreover, because the image of $\varphi^+$ must be contained in some conjugate of $\GammaL_1(2^3)$ while also containing $\GammaL_1(2^3)$, it follows that $\im(\varphi^+)=\GammaL_1(2^3)$. In particular, since $\varphi^+(\hat{\omega})^3=\varphi^+(\hat{\omega}^3)=\varphi(\hat{\omega}^3)=\hat{\omega}^9$, the image $\varphi^+(\hat{\omega})$ must be an element of $\GammaL_1(2^3)$ of order $7$. In particular, $\varphi^+(\hat{\omega})=\hat{\omega}^c$ for some $c\in\IZ$ with $\gcd(c,63)=9$. Hence, for all $x\in\IF_{2^6}$, one has
\[
\omega^{1+c}x^3+\omega^{8+c}x^{24}=\sigma(x)^{\varphi^+(\hat{\omega})}=\sigma(x^{\hat{\omega}})=\omega^4x^3+\omega^{32}x^{24},
\]
which is impossible -- the desired contradiction.

In summary, we find that $P(\omega)$ must be a Suzuki $2$-group of the form $B(3,\theta,\varepsilon)$ for a \emph{nontrivial} field automorphism $\theta$ of $\IF_{2^3}$ and some
\[
\varepsilon\in\IF_{2^3}\setminus\{\rho^{-1}+\rho^{\theta}\mid\rho\in\IF_{2^3}^\times\}=\{\omega^{9\cdot 2^i}\mid i=0,1,2\}=\{\eta^{2^i}\mid i=0,1,2\}.
\]
For the first set equality the choice of minimal polynomial for $\omega$ above matters.

We now show that all combinations of $\theta\not=\id$ and (admissible) $\varepsilon$ are possible (i.e., the corresponding groups $B(3,\theta,\varepsilon)$ are pairwise isomorphic, whence they must all be isomorphic to $P(\omega)$). Indeed, since, as observed by Higman \cite[paragraph after Theorem 1, p.~82]{Hig63a}, $B(3,\theta,\varepsilon)\cong B(3,\theta^{-1},\varepsilon)$, it suffices to show this when $\theta$ is the squaring automorphism $\chi\mapsto\chi^2$ of $\IF_{2^3}$.
Each $\zeta\in\IF_{2^6}\setminus\IF_{2^3}$ gives rise to an additive group isomorphism $\iota_{\zeta}\colon\IF_{2^3}^2\rightarrow\IF_{2^6}$, depending on the $\zeta$, with
$(\chi_1,\chi_2)\mapsto\chi_1+\chi_2\cdot\zeta$.
The following is easy to check with GAP \cite{GAP4}:
\begin{enumerate}
\item With $\zeta=\omega^{13}$, $\sigma_{\omega}\circ\iota_{\zeta}\colon\IF_{2^3}^2\rightarrow\IF_{2^6}$ is an $\IF_{2^3}$-scalar multiple of $(\chi_1,\chi_2)\mapsto\chi_1^3+\omega^9\chi_1^2\chi_2+\chi_2^3$, which is, up to a component swap, the \enquote{squaring function} of $B(3,\theta,\omega)$ specified by Higman. This shows that $G\cong B(3,\theta,\omega^9)$.
\item Similarly, with $\zeta=\omega^{44}$, $\sigma_{\omega}\circ\iota_{\zeta}$ becomes an $\IF_{2^3}$-scalar multiple of the \enquote{squaring function} of $B(3,\theta,\omega^{18})$, proving that $G\cong B(3,\theta,\omega^{18})$.
\item Finally, with $\zeta=\omega^{25}$, $\sigma_{\omega^2}\circ\iota_{\zeta}$ turns into an $\IF_{2^3}$-scalar multiple of the \enquote{squaring function} of $B(3,\theta,\omega^{36})$, which shows that $G\cong B(3,\theta,\omega^{36})$.
\end{enumerate}

\section{Appendix}\label{app}

\subsection{Proof of Proposition~\ref{4presProp}}\label{app1}

For statement (1): The power-commutator presentation $\P$ is consistent
by~\cite[Theorem~9.22]{HEO05a} and the remark at the bottom of page 362.
Consistency means that there all the relations are consequences of those
in~$\P$, so that elements of the group $G_{\P}$ represented by $\P$
have a normal form
$x_1^{\delta_1}\cdots x_m^{\delta_m}y_1^{\eps_1}\cdots y_m^{\eps_n}$
with $\delta_i,\eps_j\in\{0,1\}$.

For statement (2): That $\sigma_{\P}$ and $[\,,]_{\P}$ are well-defined through their implicit definitions is easy to see by computing in the above model of $G_{\P}$, where $w_1(u)_{G_{\P}}=(u,0)$ and $w_2(v)_{G_{\P}}=(0,v)$ for all $u\in\IF_2^m$ and all $v\in\IF_2^n$. Moreover, using that $G_{\P}$ is nilpotent of class at most $2$ with $G'=\Z(G)$ of exponent $2$, one sees that $\sigma_{\P}(u)$ is the unique value of the squaring function $x\mapsto x^2$ of $G_{\P}$ on the coset $\{u\}\times\IF_2^n$, and $[u_1,u_2]_{\P}$ is the unique value of the commutator function $(x,y)\mapsto [x,y]=x^{-1}y^{-1}xy$ of $G_{\P}$ on the coset product $(\{u_1\}\times\IF_2^n)\times(\{u_2\}\times\IF_2^n)$. This also entails that $[\,,]_{\P}$ is biadditive and alternating, and the identity $[u_1,u_2]_{\P}=\sigma(u_1+u_2)+\sigma(u_1)+\sigma(u_2)$ follows from the fact that $G_{\P}$ satisfies the identity $[x,y]=(xy)^2x^2y^2$.

For statement (3): Denoting by $e_i$, for $i=1,\ldots,m$, the $i$-th
standard basis vector of $\IF_2^m$, let $\P'$ be the consistent
power-commutator presentation of the form specified in statement (1)
with respect to the choices $\sigma_i\coloneq\sigma(e_i)$, for
$i=1,\ldots,m$, and
$\pi_{i,j}\coloneq\sigma(e_i+e_j)+\sigma(e_i)+\sigma(e_j)$, for $1\leq
i<j\leq m$. Then $[\,,]_{\P'}=[\,,]$, as these two functions agree on the
arguments $(e_i,e_j)$ for $1\leq i<j\leq m$ and they are both
alternating and biadditive. We claim that also $\sigma_{\P'}=\sigma$;
indeed, we will show by induction on the number of summands $e_i$
needed to express $u\in\IF_2^m$ that $\sigma_{\P}(u)=\sigma(u)$:
First, note that by definition of $\sigma_{\P'}$, $\sigma_{\P'}(0)=0$,
whereas $\sigma(0)=0$ holds since
$\sigma(0)=\sigma(0+0)=\sigma(0)+\sigma(0)+[0,0]$ and $[\,,]$ is
biadditive. Moreover, by construction, $\sigma$ and $\sigma_{\P'}$
agree on the elements of the standard basis $\{e_i\mid i=1,\ldots,m\}$.
The induction step, use that for all $u_1,u_2\in\IF_2^m$ with
$\sigma_{\P}(u_i)=\sigma(u_i)$ for $i=1,2$,~one~has
\[
\sigma_{\P}(u_1+u_2)=\sigma_{\P}(u_1)+\sigma_{\P}(u_2)+[u_1,u_2]_{\P}=\sigma(u_1)+\sigma(u_2)+[u_1,u_2]=\sigma(u_1+u_2).
\]

For statement (4): It is immediate to check that the specified map respects the defining relations of $\P$ and thus extends to an \emph{endo}morphism $\chi$ of $G_{\P}$. Moreover, by definition, $\chi$ maps each coset of $\langle y_1,\ldots,y_n\rangle_{G_{\P}}$ to itself, so $\ker(\chi)\leq\langle y_1,\ldots,y_n\rangle_{G_{\P}}$. But also by definition, $\chi$ acts as the identity on $\langle y_1,\ldots,y_n\rangle_{G_{\P}}$, and so $\ker(\chi)$ is trivial, i.e., $\chi$ is an automorphism of $G_{\P}$, which is central since $\langle y_1,\ldots,y_n\rangle_{G_{\P}}\leq\Z(G_{\P})$. It is easy to check that these automorphisms $\chi$ form a subgroup of $\Aut(G_{\P})$.

For statement (5): That $C_{\P}$ is the kernel of $\gamma$ follows from the fact that $C_{\P}$ consists just of those automorphisms of $G_{\P}$ that map each coset of $\langle y_1,\ldots,y_n\rangle_{G_{\P}}$ to itself (the surjectivity of $[\,,]_{\P}$ implies that all such automorphisms must fix $\langle y_1,\ldots,y_n\rangle_{G_{\P}}=G_{\P}'$ element-wise). By definition, $\gamma_1(\alpha)$ is the matrix from $\GL_m(2)$ that encodes the (linear) induced action of $\alpha$ on $G_{\P}/G_{\P}'\cong\IF_2^m$ with the respect to the $\IF_2$-basis $(x_1)_{G_{\P}}G_{\P}',\ldots,(x_m)_{G_{\P}}G_{\P}'$, and $\gamma_2(\alpha)$ is the matrix from $\GL_n(2)$ that encodes the (linear) restricted action of $\alpha$ on $G_{\P}'\cong\IF_2^n$ with respect to the $\IF_2$-basis $(y_1)_{G_{\P}},\ldots,(y_n)_{G_{\P}}$. Hence $\sigma_{\P}\circ\gamma_1(\alpha)=\gamma_2(\alpha)\circ\sigma_{\P}$ follows from the facts that $\sigma_{\P}$ encodes (with respect to the same bases) the squaring function of $G_{\P}$ and that $\alpha$ commutes with that squaring function. Conversely, given a matrix pair $(\gamma_1,\gamma_2)\in\GL_m(2)\times\GL_n(2)$ such that $\sigma_{\P}\circ\gamma_1=\gamma_2\circ\sigma_{\P}$, the function on the model $\IF_2^m\times\IF_2^n$ of $G_{\P}$ from the proof of statement (1) defined via $(u,v)\mapsto(u^{\gamma_1},v^{\gamma_2})$ is an automorphism $\alpha$ of $G_{\P}$ (it acts as an automorphism on $G_{\P}'$, commutes with the squaring function of $G_{\P}$ by definition and thus, using that $G_{\P}$ satisfies the identity $[x,y]=(xy)^2x^2y^2$, also commutes with the commutator function of $G_{\P}$, whence it respects the defining relations of $\P$) and satisfies $\gamma_i(\alpha)=\gamma_i$ for $i=1,2$ by definition.

For statement (6): Since $S_{\P}$ is a subgroup of $A_{\P}\times B_{\P}$, it suffices to show that if $(1,b)\in S_{\P}$, then $b=1$. This is equivalent to the assertion that if an automorphism $\alpha$ of $G_{\P}$ acts trivially modulo $G_{\P}'$, then the restriction of $\alpha$ to $G_{\P}'$ is also trivial, which is clear since $G_{\P}'=\Z(G_{\P})$. That the thus well-defined function $\varphi_{\P}$ is a homomorphism $A_{\P}\rightarrow B_{\P}$ follows from the fact that $S_{\P}$ is a subgroup of~$A_{\P}\times B_{\P}$.

\subsection{Graph automorphisms and point stabilizers in \texorpdfstring{$\Sp_4(\kern-0.7pt2^b\kern-0.7pt)$}{Sp4(2b)}}\label{app2}

In the sequel, we view $\Sp_4(2^b)$ as the set $\{X\in\GL_4(2^b)\mid X^{\tr}JX=J\}$ where
\[
J=\scalemath{0.8}{\begin{pmatrix}0 & 0 & 0 & 1 \\ 0 & 0 & 1 & 0 \\ 0 & 1 & 0 & 0 \\ 1 & 0 & 0 & 0\end{pmatrix}}.
\]
Moreover, we view the semilinear symplectic group $\GammaSp_4(2^b)$ as
an index $2$ subgroup of $\Aut(\Sp_4(2^b))$ acting naturally on
$\Sp_4(2^b)$. The following
result is used in~Section~\ref{sec3}:

\begin{proposition*}\label{graphProp}
Let $b\in\IN^+$, and let $\alpha\in\Aut(\Sp_4(2^b))\setminus\GammaSp_4(2^b)$ (i.e., $\alpha$ involves some non-trivial graph automorphism of $\Sp_4(2^b)$). Then for each $v\in\IF_{2^b}^4\setminus\{0\}$, the elements of $\Stab_{\Sp_4(2^b)}(v)^{\alpha}$ do not have a common nonzero fixed point in $\IF_{2^b}^4$.
\end{proposition*}

\begin{proof}[Proof of the Proposition]
As images of point stabilizers in $\Sp_4(2^b)$ under automorphisms from $\GammaSp_4(2^b)$ are still point stabilizers, and since $\GammaSp_4(2^b)$ is of index $2$ in $\Aut(\Sp_4(2^b))$, it suffices to show that there exists such an $\alpha$ in $\Aut(\Sp_4(2^b))\setminus\GammaSp_4(2^b)$. Let us choose $\alpha$ to be the following graph automorphism of $\Sp_4(2^b)$ described in \cite[p.~235]{Sol77a}, whose construction is originally due to Wong:
\begin{align*}
&\alpha:\Sp_4(2^b)\rightarrow\Sp_4(2^b), \\
&\scalemath{0.8}{\begin{pmatrix}
\beta_{11} & \beta_{12} & \beta_{13} & \beta_{14} \\
\beta_{21} & \beta_{22} & \beta_{23} & \beta_{24} \\
\beta_{31} & \beta_{32} & \beta_{33} & \beta_{34} \\
\beta_{41} & \beta_{42} & \beta_{43} & \beta_{44}
\end{pmatrix}
\mapsto 
\begin{pmatrix}
\beta_{12}\beta_{21}+\beta_{11}\beta_{22} & \beta_{13}\beta_{21}+\beta_{11}\beta_{23} & \beta_{14}\beta_{22}+\beta_{12}\beta_{24} & \beta_{14}\beta_{23}+\beta_{13}\beta_{24} \\
\beta_{12}\beta_{31}+\beta_{11}\beta_{32} & \beta_{13}\beta_{31}+\beta_{11}\beta_{33} & \beta_{14}\beta_{32}+\beta_{12}\beta_{34} & \beta_{14}\beta_{23}+\beta_{13}\beta_{23} \\
\beta_{22}\beta_{41}+\beta_{21}\beta_{42} & \beta_{23}\beta_{41}+\beta_{21}\beta_{43} & \beta_{24}\beta_{42}+\beta_{22}\beta_{44} & \beta_{24}\beta_{43}+\beta_{23}\beta_{44} \\
\beta_{32}\beta_{41}+\beta_{31}\beta_{42} & \beta_{33}\beta_{41}+\beta_{31}\beta_{43} & \beta_{34}\beta_{42}+\beta_{32}\beta_{44} & \beta_{34}\beta_{43}+\beta_{33}\beta_{44}
\end{pmatrix}.}
\end{align*}
Since $\Sp_4(2^b)$ acts transitively on $\IF_{2^b}^4\setminus\{0\}$, we can also restrict $v$ to the concrete choice $v\coloneq e_1=(1,0,0,0)$, the first standard basis vector of $\IF_{2^b}^4$.

Let us first work out a parametrization for the elements of $S\coloneq\Stab_{\Sp_4(2^b)}(e_1)$. Any $g\in S$ satisfies $e_1^g=e_1$, $e_i^g\in e_1^{\perp}=\IF_{2^b}e_2+\IF_{2^b}e_3+\IF_{2^b}e_4$ for $i\in\{2,3\}$ and $e_4^g\in e_4+e_1^{\perp}$ (so the fourth coordinate of $e_4^g$ is $1$). Hence one can write each $g\in S$~as
\[
g=\scalemath{0.8}{\begin{pmatrix}1 & 0 & 0 & 0 \\ x_{21} & x_{22} & x_{23} & 0 \\ x_{31} & x_{32} & x_{33} & 0 \\ x_{41} & x_{42} & x_{43} & 1\end{pmatrix}}
\]
for suitable $x_{ij}\in\IF_{2^b}$, $i\in\{2,3,4\}$, $j\in\{1,2,3\}$. However, not all such choices of $x_{ij}$ result in an element of $S$; to get an actual parametrization of the elements of~$S$, we need to solve the matrix equation
\begin{equation}\label{matrixEq}
\scalemath{0.8}{\begin{pmatrix}1 & x_{21} & x_{31} & x_{41} \\ 0 & x_{22} & x_{32} & x_{42} \\ 0 & x_{23} & x_{33} & x_{43} \\ 0 & 0 & 0 & 1\end{pmatrix}J\begin{pmatrix}1 & 0 & 0 & 0 \\ x_{21} & x_{22} & x_{23} & 0 \\ x_{31} & x_{32} & x_{33} & 0 \\ x_{41} & x_{42} & x_{43} & 1\end{pmatrix}=J.}
\end{equation}

Multiplying the matrices on the left-hand side of Formula~(\ref{matrixEq}),
we get
\begin{align*}
&
\scalemath{0.8}{\begin{pmatrix}1 & x_{21} & x_{31} & x_{41} \\ 0 & x_{22} & x_{32} & x_{42} \\ 0 & x_{23} & x_{33} & x_{43} \\ 0 & 0 & 0 & 1\end{pmatrix}J\begin{pmatrix}1 & 0 & 0 & 0 \\ x_{21} & x_{22} & x_{23} & 0 \\ x_{31} & x_{32} & x_{33} & 0 \\ x_{41} & x_{42} & x_{43} & 1\end{pmatrix}}=
\scalemath{0.8}{\begin{pmatrix}x_{41} & x_{31} & x_{21} & 1 \\ x_{42} & x_{32} & x_{22} & 0 \\ x_{43} & x_{33} & x_{23} & 0 \\ 1 & 0 & 0 & 0\end{pmatrix}\begin{pmatrix}1 & 0 & 0 & 0 \\ x_{21} & x_{22} & x_{23} & 0 \\ x_{31} & x_{32} & x_{33} & 0 \\ x_{41} & x_{42} & x_{43} & 1\end{pmatrix}}\\
&=\scalemath{0.8}{\begin{pmatrix}0 & x_{42}+x_{22}x_{31}+x_{21}x_{32} & x_{43}+x_{23}x_{31}+x_{33}x_{21} & 1 \\ x_{42}+x_{21}x_{32}+x_{22}x_{31} & 0 & x_{23}x_{32}+x_{22}x_{33} & 0 \\ x_{43}+x_{21}x_{33}+x_{23}x_{31} & x_{22}x_{33}+x_{32}x_{23} & 0 & 0 \\ 1 & 0 & 0 & 0\end{pmatrix}}.
\end{align*}

By comparing coefficients with $J$, the right-hand side of Formula (\ref{matrixEq}), we arrive at the following system of equations (with redundant equations omitted):
\[
\scalemath{0.8}{\begin{cases}\text{I:} & x_{42}=x_{21}x_{32}+x_{22}x_{31} \\ \text{II:} & x_{43}=x_{21}x_{33}+x_{23}x_{31} \\ \text{III:} & x_{22}x_{33}+x_{32}x_{23}=1.\end{cases}}
\]
Hence, the general form of an element of $S=\Stab_{\Sp_4(2^b)}(e_1)$ is
\[
\scalemath{0.8}{\begin{pmatrix}
1 & e & f & g \\
0 & a & b & be+af \\
0 & c & d & cf+de \\
0 & 0 & 0 & 1
\end{pmatrix}}
\]
where $a,\dots,g\in\IF_{2^b}$ and $ad+bd=1$. Thus, the general form of an element of $S^{\alpha}$ is
\[
\scalemath{0.8}{\begin{pmatrix}
a & b & be^2+aef+ag & bef+af^2+bg \\
c & d & cef+de^2+cg & cf^2+def+dg \\
0 & 0 & a & b \\
0 & 0 & c & d
\end{pmatrix}\eqcolon M(a,b,c,d,e,f,g)}
\]
with $a,\dots,g\in\IF_{2^b}$ and $ad+bc=1$. We now argue that if a vector $x=(x_1,x_2,x_3,x_4)$ in $\IF_{2^b}^4$ is fixed by all such matrices $M(a,\ldots,g)$, then it must be the zero vector. Indeed, assume first that $(x_3,x_4)\not=(0,0)$. Since $\SL_2(2^b)$ acts transitively on $\IF_{2^b}^2\setminus\{0\}$, we can fix $a,b,c,d\in\IF_{2^b}$ with $ad+bc=1$ such that $(x_3,x_4)\left(\begin{smallmatrix}a & b \\ c & d\end{smallmatrix}\right)\not=(x_3,x_4)$, and then for any choice of $e,f,g\in\IF_{2^b}$, we have (considering the last two coordinates of both sides) $xM(a,\ldots,g)\not=x$, a contradiction. So we may assume that $(x_3,x_4)=(0,0)$. But then an analogous argument shows that necessarily $(x_1,x_2)=(0,0)$ as well, so that $x$ is indeed the zero vector, as asserted.
\end{proof}

\section*{Acknowledgement}
The authors thank the referee and the editor for their helpful suggestions.



\end{document}